\documentclass[10 pt, pdftex, reqno]{amsart}
\usepackage[a4paper, margin=1.5in]{geometry}
\usepackage{booktabs}
\usepackage{colortbl}
\usepackage{arydshln}
\usepackage{float}
\usepackage{yfonts}

\usepackage{amsmath,amsthm,amssymb,amscd, graphicx, color}
\usepackage[all]{xy}
\usepackage{centernot}
\usepackage{mathtools}
\usepackage{stmaryrd}
\usepackage[shortlabels]{enumitem}
\theoremstyle{plain}

\newtheorem{theorem}{Theorem}[]
\newtheorem{proposition}[theorem]{Proposition}
\newtheorem{lemma}[theorem]{Lemma}

\theoremstyle{definition}

\newtheorem{definition}[theorem]{Definition}

\newtheorem{notation}[theorem]{Notation}

\newtheorem{remark}[theorem]{Remark}

\theoremstyle{remark}

\newcommand{\oo}{\mathcal{O}}
\newcommand{\QQ}{\mathbb{Q}}
\newcommand{\RR}{\mathbb{R}}
\newcommand{\PP}{\mathbb{P}}

\begin{document}

%\usepackage[utf8x]{inputenc}
%\usepackage[ansinew]{inputenc}
%\usepackage[cyr]{aeguill}
%\usepackage{ae,aecompl,aeguill}
%\usepackage[french]{babel}
%\usepackage[french]{babel}
%\usepackage[latin1]{inputenc}
%\usepackage[T1]{fontenc}
%-------------------------------------------

\pagestyle{plain}
\bibliographystyle{hplain}

\title{Anticanonical divisors on Fano fourfolds}

\author{Liana Heuberger}
\address{Liana Heuberger, IMJ, Universit\'e Pierre et Marie Curie, 4 Place Jussieu, 75005 Paris, France}

\email{liana.heuberger@imj-prg.fr}

\begin{abstract}
 
Let X be a Fano manifold. A result by Shokurov states that in dimension three the linear system $|-K_X|$ is non empty and a general element $D \in |-K_X|$ is smooth. In dimension four, one can construct Fano varieties $X$ such that every such $D$ is singular, however we show it has at most terminal singularities. We then determine an explicit local expression for these singular points. 

\end{abstract}

\maketitle

\section{Introduction}

A Fano manifold is a projective manifold $X$ with ample anticanonical divisor $-K_X$. In this paper we characterize the singularities of general divisors in the complete linear system $|-K_X|$, in the case where the dimension of $X$ is equal to four. We sometimes refer to these objects as \emph{general elephants}, a terminology introduced by Miles Reid in \cite{YPG}.

One of the central results on the anticanonical system in dimension three is the following:

\begin{theorem}{\cite{Sho}}
\label{shok}
Let $X$ be a smooth Fano threefold. Then the anticanonical system $|-K_X|$ is not empty and a general divisor $D\in|-K_X|$ is a smooth $K3$ surface.
\end{theorem}

A key observation in the proof of this theorem is that if $\mathrm{Bs}|-K_X|\neq \emptyset$, then it is isomorphic to $\PP^1$. The result was a fundamental step in the classification of Fano threefolds of Picard rank one, which, together with the work of Mori and Mukai in \cite{MorMuk}, provided a complete classification of Fano manifolds in dimension three. 

The methods used in proving Threorem \ref{shok} are not generalizable in higher dimensions since they rely on the geometry of $K3$ surfaces. Moreover, its statement does not hold in its current form if $X$ is a fourfold, as shown in \cite[Ex. 2.12]{HV11}, an example where each general $D\in |-K_X|$ is not even $\QQ$-factorial. A correct generalization would be that in which the type of singularities appearing on such a threefold $D$ become smooth if $D$ is a surface.

In this paper we prove the following result:

\begin{theorem}
\label{thm 1}
Let $X$ be a four-dimensional Fano manifold and let $D\in |-K_X|$ be a general divisor. Then $D$ has at most terminal singularities.  
\end{theorem}

Theorem \ref{thm 1} is indeed the natural generalization of Theorem \ref{shok} in dimension four, and it relies on and improves the following existing results concerning the geometry of a general elephant:

\begin{proposition}{\cite[Thm.1.7]{HV11}\cite[Thm 5.2]{Kawamata1}}
\label{known}
Let $X$ be a four-dimensional Fano manifold and $D\in |-K_X|$ be a general divisor. We have the following:
\begin{enumerate}[1)]
\item $h^0(X,-K_X)\geq 2$.
\item $D$ is irreducible. In particular, this implies that $\textup{Bs}|-K_X|$ is at most a surface.
\item \label{canonical}  $D$ has at most isolated canonical singularities.
\end{enumerate}
\end{proposition}

Although terminal Gorenstein singularities of threefolds are a well understood class, the statement of Theorem \ref{thm 1} may be further refined: while there only exists a finite number of deformation families of Fano varieties in dimension four, the classification of threefolds with isolated terminal singularities provides an infinite amount of examples.

\begin{theorem}
\label{thm 2}
Let $X$  be a four-dimensional Fano manifold and let $D\in |-K_X|$ be a general divisor. Then the singularities of $D$ are locally analytically given by \[x_1^2+x_2^2+x_3^2+x_4^2=0 \ \text{ or } \ x_1^2+x_2^2+x_3^2+x_4^3=0.\]
\end{theorem}

In particular, this result is consistent with the case of \cite[Ex. 2.12]{HV11}, which is a singularity of the first type. We do not yet know of any examples of the second type of singularity on a general elephant.

For the proof we need to consider the geometry of the ambient space together with the fact that the threefolds belong to the same linear system inside it. We begin the analysis of these isolated terminal points by separating the discussion into two cases relative to the geometry of all general elements in $|-K_X|$: fixed and moving singularities. Specifically, either a point $x\in \mathrm{Bs}|-K_X|$ is singular on all the general elephants or there exists a subvariety $V\subset \mathrm{Bs}|-K_X|$ of strictly positive dimension along which these singularities move. We further separate both of these cases according to the rank of the degree two part of a local expression of $D$ and obtain the result in Theorem \ref{thm 2}. 

The fundamental tool in the proofs of both Theorem \ref{thm 1} and \ref{thm 2} is the following inequality:\begin{equation*}\label{introairi} a_i+1\geq 2r_i\end{equation*} where, given a resolution $\mu:X'\to X$ of $\mathrm{Bs}|-K_X|$, we denote by $a_i$ the discrepancies of each exceptional divisor $E_i$ with respect to $(X,0)$ and by $r_i$ the coefficients of $E_i$ in $\mu^*D$ (cf. Notation \ref{notation}). We prove this in Proposition $\ref{multideal}$ with techniques using singularities of pairs and multiplier ideals. The inequality straightforwardly implies the terminality result for all cases except $a_i=r_i=1$, which we show does not in fact occur. It also allows us to systematically eliminate most of the cases leading to the statement of Theorem \ref{thm 2}, providing the necessary liaison between the geometry of $X$ and that of $D$. The general strategy of the proof is to explicitly build a sequence of blow-ups, starting from a center in $X$ containing either the fixed singularities or the subvariety $V$, that ultimately contradicts the above inequality and provides a contradiction. 

\subsection*{Acknowledgements} I would like to thank my Ph.D advisor, Prof. Andreas H\"{o}ring, for suggesting this problem as well as for his patience, constant support and encouragement. 
 
\section{Terminality}
\label{terminality}

\begin{definition}
Let $X$ be a normal, integral scheme and $D=\sum d_iD_i$ an $\RR$-divisor such that $K_X+D$ is $\RR$-Cartier. Let $\mu:X'\to X$ be a birational morphism, with $X'$ normal. If we write \[K_{X'}=\mu^*(K_X+D)+\sum a(E,X,D)E,\] where $E\subset X'$ are distinct prime divisors and $a(E,X,D)\in \RR$, the discrepancy of the pair $(X,D)$ is: \[\mathrm{discrep}(X,D):=\inf\limits_E\, \{ a(E,X,D)\,|\, E \text{ is exceptional with nonempty center on }X\}.\] 

A pair $(X,D)$ is \emph{canonical} if $\mathrm{discrep}(X,D)\geq 0$ and \emph{klt} (Kawamata log terminal) if $\mathrm{discrep}(X,D)>-1$.

We say that $X$ has \emph{terminal} singularities if $\mathrm{discrep}(X,0)>0$.
\end{definition}

\begin{notation}\label{notation}
Let $X$ be a projective manifold. Given a resolution $\mu: X'\rightarrow X$ of the base locus of $|-K_X|$ such that $E= \sum\limits_{i = 1}^m {E_i }$ is its exceptional locus and $D\in |-K_X|$ is a general element, we then write:

\begin{itemize}
\renewcommand{\labelitemi}{$\bullet$}
\item $|\mu^*D|= |D'|+ \sum\limits_{i=1}^m {r_iE_i}$, where $D'$ is the strict transform of D and $r_i > 0$ for $i\in \{1\ldots m\}$,
\item $K_{X'} = \mu^* K_X + \sum\limits_{i=1}^m {a_iE_i}$,  where $a_i > 0$, for all $i\in \{1\ldots m\}$.
\end{itemize}
\end{notation}
These coefficients will be extensively used throughout our proofs.

\begin{remark}
\label{snc}
By Bertini's Theorem, if $|L|$ is a linear system on a smooth variety $X$ and $D_1, D_2\in |L|$ are two general elements, then the divisor $D_1+D_2$ is SNC outside $\mathrm{Bs}|L|$.
\end{remark}

\begin{definition}{\cite[Def. 9.2.10]{Laz}}
Let $|L|$ be a non-empty linear series on a smooth complex variety $X$ and let $\mu:X^\prime \to X$ be a log resolution of $|L|$, with \[\mu^*|L|=|W|+F,\] where $F+\textup{exc}(\mu)$ is a divisor with SNC support and $W\subseteq H^0(X',\oo_{X'}(\mu^*L-F))$. Given a rational number $c>0$, the \emph{multiplier ideal} $\mathcal{J}(c\cdot |L|)$ corresponding to $c$ and $|L|$ is \[\mathcal{J}(c\cdot |L|)=\mathcal{J}(X, c\cdot |L|)=\mu_\ast\oo_{X^\prime}(K_{X^\prime / X}-[c\cdot F]). \]
\end{definition}

\begin{proposition}
\label{multideal}
Let $X$ be a four-dimensional Fano manifold. Then for every $c<2$ we have that \[\mathcal{J}(c\, |-K_X|)=\oo_X.\] In terms of the coefficients in Notation \ref{notation}, this is equivalent to \begin{equation}  \label{airi}
\forall \:  i\in \{1\ldots m\} \: : \: a_i+1\geq 2r_i. \end{equation}

\end{proposition}

\begin{proof}
Arguing by contradiction, we suppose there exists a rational number $c<2$ such that $\mathcal{J}(c|-K_X|)\subsetneq \oo_X$. By \cite[Prop.9.2.26]{Laz}, this is equivalent to the fact that the pair $\left(X,c\,\frac{D_1+D_2}{2}\right)$ is not klt for two general divisors $D_1,\, D_2\in |-K_X|$. Take $c_0<c$ to be the log canonical threshold of $\left(X,\frac{D_1+D_2}{2}\right)$, thus producing a properly log canonical pair $\left(X,c_0\frac{D_1+D_2}{2}\right)$ which admits a minimal log canonical center, denoted in what follows by $C$. 

As by Remark \ref{snc} the divisor $D_1+D_2$ has simple normal crossings outside the base locus of $|-K_X|$, the identity map is a log resolution of the pair $\left(X\setminus \mathrm{Bs}|-K_X|, c_0\frac{D_1+D_2}{2}\right)$. Since $\frac{c_0}{2}<1$ we deduce that this pair is klt, which shows that the log canonical center $C$ must be included in the base locus. As the dimension of $\mathrm{Bs}|-K_X|$ is at most two by Proposition \ref{known}, then $C$ is also at most a surface.

Since $c_0<2$, we can apply \cite[Thm. 2.2]{Fundamental} to the pair $\left(X,c_0\frac{D_1+D_2}{2}\right)$ together with the anticanonical divisor. As the divisor $-K_X-(K_X+c_0\frac{D_1+D_2}{2})\sim (2-c_0)(-K_X)$ is ample, we have obtained a surjective map $$H^0(X,\mathcal{O}_X(-K_X))\twoheadrightarrow H^0(C,\oo_C(-K_X)).$$ This map is the zero map since $C$ is contained in $\mathrm{Bs}|-K_X|$, and in order to obtain a contradiction we show that the target is nontrivial.

Indeed, using the minimality of $C$ and the Kawamata Subadjunction formula \cite[Thm.1.2]{FuGo}, there exists an effective $\mathbb{Q}$-divisor $B\subset C$ such that $$\left(K_X+c_0\frac{D_1+D_2}{2}\right)|_C \sim_{\mathbb{Q}} K_C+B$$ and the pair $(C,B)$ is klt. This provides the ingredients to apply Kawamata's Nonvanishing Theorem \cite[Thm.3.1]{Kawamata1} to the pair $(C,B)$ and the divisor $-K_X|_C$. As the divisor $$-K_X|_C-(K_C+B)\sim_{\mathbb{Q}} -K_X|_C-\left(K_X+c_0\frac{D_1+D_2}{2}\right)|_C\sim_{\mathbb{Q}} (2-c_0)(-K_X|_C)$$ is ample, it follows that $H^0(C,\oo_C(-K_X))\neq \emptyset$. 
\end{proof}

This statement immediately implies our first result:

\begin{proof}[Proof of Theorem \ref{thm 1}]

Let $\mu$ be a resolution as in Notation \ref{notation}. The adjunction formula for a general elephant gives us:
\begin{equation}\label{discr}K_{D'}=(\mu|_{D'})^* K_D+\sum\limits_{i=1}^{m}(a_i-r_i)(E_i\cap D') \end{equation}  
 which means that the discrepancy of $(D,0)$ is $\inf\limits_{i}  \{a_i-r_i \ | E_i \text{ is } \mu\text{-exceptional} \}$. As by Proposition \ref{known},\ref{canonical} we already know that this is non-negative, the aim of what follows is to show that the discrepancy of this pair is non-zero. Note that since we have considered a log resolution, the intersection $E_i\cap D'$ is reduced for all $i\in\{1\ldots m\}$.

We further argue that the inequality in condition \eqref{airi} is sufficient in order to obtain terminality.  Indeed, the only case in which this doesn't imply $a_i-r_i>0$ is if both $a_i$ and $r_i$ are equal to one. 

Since the coefficients do not depend on the choice of the resolution, we can assume that we have been working with one in which all blow-ups were made along smooth centers.

\textit{Claim:} We can only obtain $a_i=1$ for a certain $i\in \{ 1\ldots m\}$ if $\textup{codim}_X\mu(E_i)=2$.

Without loss of generality, we can assume we obtained this coefficient by doing the very last blow-up of the resolution, denoted by $\mu_m$: \[\xymatrix@ !{& X'=X_m \ar[r]^{\mu_m}& X_{m-1} \ar[r]^{\psi} &X}\] where $\mu=\psi \circ\mu_m.$ 
We have that
 \[a_m=\lambda+\sum\limits_{i=1}^{m-1}a_i \nu_i\] where $\lambda=\mathrm{codim}_{X_{m-1}}\mu_m(E_m)-1 \geq 1$ and $\nu_i>0$ if and only if $\mu_m(E_m)\subset E_i$. Having $a_m=1$ implies $\lambda=1$ and $\nu_i=0 \ \forall \, i$, which proves the claim since the former condition shows that $\mu_m(E_m)$ is exactly of codimension two in $X_{m-1}$, while the latter signifies that $\psi$ does not contract $\mu_m(E_m)$.

As $\mbox{codim}\psi(E_i)=2$ implies that the intersection $E_i\cap D'$ is not $\mu|_{D'}$-exceptional, this divisor does not contribute to the discrepancy of the pair $(D,0)$ as computed in \eqref{discr}. Together with condition \eqref{airi} this proves that the discrepancy can never be zero, therefore a general elephant $D$ has at most terminal singularities.
\end{proof} 

\section{Separating strict transforms}

We now provide the set-up for the discussion describing the local equations of these singular points.

\begin{notation}
Let $|L|$ be a linear system on a projective manifold $X$ and let $D$ be an effective prime divisor on $X$. We denote by $|L|_D$ the linear system on $D$ given by the image of the restriction morphism: $$H^0(X,\oo_X(L))\rightarrow H^0(D,\oo_D(L)).$$ 
We obtain a linear system on $D$ that will not only be determined by the intrinsic properties of $D$, but which fundamentally depends on the behavior of $|L|$ on $X$. An immediate consequence of this is: \begin{equation} \label{baselocus} \mathrm{Bs}|L|_D=\mathrm{Bs}|L|\cap D.\end{equation}
\end{notation}

\begin{notation}
Given a linear system $|L|$ on a projective variety $X$ and a point $x\in \mathrm{Bs}|L|$,  denote by $|L|_x$ the following closed subset of $|L|$: \begin{align*}|L|_x:=\{D\in |L| \; \vert \; x\in D_{sing}\}. \end{align*}
\end{notation}

\begin{lemma}[Tangency lemma]
\label{tangency lemma}
Let $X$ be a projective manifold and let $|L|$ be a linear system on $X$. Let $x\in X$ be a point in $\mathrm{Bs}|L|$ such that $\textup{codim}_L|L|_x=1$. Then the tangent spaces at $x$ of all divisors in $|L|\setminus |L|_x$ coincide.

\end{lemma}
\begin{proof}
Let $D_1, \, D_2 \in |L|\setminus |L|_x$ be two divisors and denote by $P:=\langle D_1, D_2\rangle$ the pencil that they generate. 

As $|L|$ is a projective space, any intersection between a codimension one subset and a line is non-empty, thus there exists a divisor $D\in |L|_x\cap P$. If $f_i$ are local equations of $D_i$ around $x$, there exist two scalars $\lambda, \eta \in \mathbb{C}$ such that $D$ is given by: \begin{align*}f=\lambda f_1+\eta f_2.\end{align*} We differentiate and obtain \begin{align*}\nabla f=\lambda \nabla f_1+\eta \nabla f_2,\end{align*} and as $D$ is singular at $x$, the left hand side vanishes at this point. On the other hand, both $\nabla f_1(x)$ and $\nabla f_2(x)$ are nonzero, thus they must be proportional. This is the same as saying that the tangent spaces of $D_1$ and $D_2$ coincide at the point $x$.
\end{proof}

\begin{remark}
\label{post-tangency}
Throughout this article, we use the lemma above in two particular cases: 
\begin{itemize}
\item there exists a curve $C\subset X$ such that for all $D\in |L|$ there exists a point $x\in D_{sing}\cap C$ and the union of these points is dense in $C$. 
\item there exists a surface $S\subset X$ such that for all $D\in |L|$ there exists a curve $C\subset D_{sing}\cap S$ and the union of all such curves is dense in $S$. 
\end{itemize}
We show that the first case satisfies the hypotheses of Lemma \ref{tangency lemma}. The second case is similar.

Let $|L|^0$ be the Zariski open set in $|L|$ such that for all $D\in |L|^0$ we have $D_{sing}\subset \mathrm{Bs}|L|$. Denote by $\mathcal{U}=\{ (D,x)\,\vert\, D \in |L|^0, \, x\in D_{sing} \}$ the universal family over $|L|^0$ and take $p_1$ and $p_2$ to be the projections on the first and second factor respectively.

Since every $D\in |L|^0$ has an isolated singularity we have that $p_1$ is a finite morphism and the fiber of $p_2$ over a point $x\in C\cap D_{sing}$ is $|L|_x$.  As $p_2$ is dominant, we obtain that \begin{align*}\mathrm{dim}|L|=\mathrm{dim}\,\mathcal{U}=\mathrm{dim}|L|_x+\mathrm{dim}C=\mathrm{dim}|L|_x+1.\end{align*}
\end{remark}

\begin{notation}
In what follows, if \[
  \xymatrix@!{
              & {X} & \ar[l]_{\mu_1}   {X_1} & \ar[l]_{\mu_2}  \ldots &  \ar[l]_{\mu_i}  {X_i}  &  \ar[l]_{\mu_{i+1}} \ldots
                    }
\]
is a sequence of blow-ups, we set $|L_0|=|-K_X|$ and we recursively define $|L_i|$ as the linear system on $X_i$ spanned by the strict transforms of general elements in $|L_{i-1}|$. The index $i$ is a good way to keep track of the level that we are on: as before, exceptional divisors of $\mu_i$ are denoted by $E_i$, while members of $|L_i|$ are denoted by $D_i$, $D_i'$, $\widetilde{D}_i$ etc.

\end{notation}

We begin to examine the local picture around a singular point of a general elephant. In order to obtain the equations in Theorem \ref{thm 2}, we construct a sequence of blow-ups contradicting condition \eqref{airi} until we are only left with the possibilities in the statement. The choice of the first blow-up depends on the nature of the singular point relative to the entire linear system $|-K_X|$. Essentially, there are two possible cases: fixed and moving singularities.

\subsection{Fixed singularities}
\label{fixedsec}

Let $X$ be a four-dimensional Fano manifold and suppose that there exists a point $x\in X$ such that for all $D\in |-K_X|$ we have $x\in D_{sing}$. The point $x$ is called a fixed singularity of the linear system $|-K_X|$.

We organize the singularities of a general elephant according to the rank of the hessian of a local equation, using the Morse Lemma for holomorphic functions:

\begin{lemma}{\cite[Thm.11.1]{AGV}}
\label{Morse}
There exists a neighborhood of a critical point where the rank of the second differential is equal to $k$, in which a holomorphic function in $n$ variables can locally analytically be written as:
$$f(x_1\ldots x_n)=x_1^2+\ldots+ x_k^2+g(x_{k+1},\ldots,x_n),$$ where the second differential of $g$ at zero is equal to zero, that is $g$ is at least of degree three in the variables $x_{k+1},\ldots, x_n$.
\end{lemma}

Here is the main result of this section:

\begin{theorem}
\label{fixed}
Let $X$ be a four-dimensional Fano manifold and suppose that there exists a point $x\in X$ such that for all $D\in |-K_X|$ we have $x\in D_{sing}$. Then around this point each general elephant is defined by an equation of one of the following two forms: $$x_1^2+x_2^2+x_3^2+x_4^2=0 \text{ or } x_1^2+x_2^2+x_3^2+x_4^3=0.$$
\end{theorem}

Throughout the proof we repeatedly use the following lemma:

\begin{lemma} \label{Bs-curve} Under the assumptions of Theorem $\ref{fixed}$, let $\mu_1:X_1\rightarrow X$ be the blow-up of $X$ at the point $x$ and $E_1$ its exceptional divisor. Denote by $|L_1|$ the linear system on $X_1$ spanned by all the strict transforms $D_1$ of general divisors $D\in|-K_X|$. Then the intersection $\mathrm{Bs}|L_1| \cap E_1$ is at most a curve.\end{lemma}

\begin{proof}
Suppose that $\mathrm{dim}(\mathrm{Bs}|L_1|\cap E_1)=3$. This implies that $E_1$ is a fixed component of each $D_1$, a contradiction since a strict transform doesn't contain the exceptional divisor.

If $\mathrm{Bs}|L_1|\cap E_1$ contains a surface $S$, let $\mu_2:X_2\rightarrow X_1$ be its blow-up and $E_2$ the unique exceptional divisor mapping onto $S$. We compute the discrepancies $a_i$ and the coefficients $r_i$, for $i=1,2$, as introduced in Notation \ref{notation}.

The dimensions of the centers give us that:
\begin{center}
$\begin{array}{c}K_{X_1}=\mu_1^*K_X+3E_1\\
K_{X_2}=\mu_2^*{K_{X_1}}+E_2+F,\end{array}$
\end{center}
where $F$ consists of other exceptional divisors not mapping onto $S$.
As $E_1$ is smooth along $S$, we have that $\mu_2^*E_1=E_1'+E_2$, thus obtaining $a_1=3$ and $a_2=4$.

The computations of the $r_i$ depend on the multiplicity of $x$ on $D$, which we denote by $m_1\geq 2$ since $D$ is singular at $x$. Set $m_2\geq 1$ to be the multiplicity of $D_1$ along $S$, then the coefficients are:
$$\mu_1^* D=D_1+m_1E_1$$
$$\mu_2^*\mu_1^*D=D_2+m_1E_1'+(m_1+m_2)E_2,$$
where $D_2\subset X_2$ is the strict transform of $D_1$. Thus $r_1=m_1$ and $r_2=m_1+m_2$. Clearly \[2r_2=2(m_1+m_2)\geq 2m_2+4\geq 6\] is strictly larger than $a_2+1=5$, hence by condition \eqref{airi} we obtain a contradiction. 
\end{proof}

\begin{definition}
We give three alternative definitions of the same concept:
\begin{itemize}
\item Let $D$ be a divisor on a smooth variety $X$ and take a point $x\in D$. The tangent cone of $D$ at $x$ is the variety $\mathrm{Spec}(\mathrm{gr}_m\oo_{D,x})$, where \[\mathrm{gr}_m\oo_{D,x}= \bigoplus\limits_{i\geq 0} m^i/m^{i+1},\] and $m$ is the maximal ideal of $\oo_{D,x}$

\item If locally we assume that $x$ is the origin and that $D$ is given by the ideal $I$, this is also the variety whose ideal is $\mathrm{in}(I)$, the initial ideal associated to $I$. Geometrically, this corresponds to the union of the tangent lines to $D$ at the point $x$.

\item Let $\mu:X_1\to X$ be a blow-up at a point $x\in X$, let $D_1$ be the strict transform of $D$ and $E$ the exceptional divisor. The tangent cone is then isomorphic to $D_1\cap E$.
\end{itemize}
\end{definition}

As we apply the Morse Lemma, the second interpretation says that the tangent cone is always given by the quadratic terms of a local equation of the general elephant. Combining this with the third interpretation will allow us to derive information on the singularities of blow-ups.

\begin{proof}[Proof of Theorem \ref{fixed}] 
Theorem \ref{thm 1} states that $D$ is terminal, and by the classification of terminal singularities (refer to \cite{Mor}, \cite{YPG} and \cite{Kollar3}) we obtain that the singularity is a point of multiplicity two on $D$. We analyze each of the cases in the Morse Lemma. Namely, we are in one of four situations corresponding to the rank of the hessian of a local expression of $D$:

\medskip

\noindent \textbf{Rank one:} By Lemma \ref{Morse}, a general member of $D\in|-K_X|$ is locally given by:
$$x_1^2+g(x_2,x_3,x_4)=0,$$ where the degree of $g$ is at least equal to three. The tangent cone $D_1\cap E_1$ is singular along the entire surface $[0:x_2:x_3:x_4]$. By Bertini's Theorem all the singularities of general divisors in the linear system $|L_1|_{E_1}$ are contained in $\mathrm{Bs}|L_1|_{E_1}$. Since $\mathrm{Bs}|L_1|_{E_1}=\mathrm{Bs}|L_1|\cap E_1$ by (\ref{baselocus}), this contradicts Lemma \ref{Bs-curve}.

\medskip

\noindent \textbf{Rank two:} In this case, the equation of a general $D\in |-K_X|$ is the following:
$$x_1^2+x_2^2+g(x_3,x_4)=0,$$ where again $g$ is of degree three or higher. The tangent cone is singular, this time along the line $l_{D_1}=[0:0:x_3:x_4]$ for a general $D_1\in |L_1|$. Note that there is no immediate contradiction, since by Lemma \ref{Bs-curve} the intersection $\mathrm{Bs}|L_1|\cap E_1$ can be a curve $C$. By Bertini's Theorem the curve $C$ contains $l_{D_1}$ as an irreducible component. Since there are only finitely many irreducible components of $C$ and an infinite numbers of strict transforms, the tangent cones are in fact singular along the same component $C_1$, that is the line $l_{D_1}$ is independent of the choice of $D$.

First suppose that a general $D_1\in|L_1|$ is singular along $C_1$ and denote its multiplicity by $m\geq2$. Blowing up $X_1$ along $C_1$ and arguing as in the proof of Lemma \ref{Bs-curve} we obtain the coefficients $a_2=5$ and $r_2=m+2$, which contradict condition \eqref{airi}. 

Thus both $D_1$ and $E_1$ are smooth at the generic point of $C_1$ and their intersection is a surface $S$ which is singular along $C_1$. This is precisely the previously mentioned tangent cone. As $S$ is singular and both $J_{E_1}|_{C_1}$ and $J_{D_1}|_{C_1}$ are of maximal rank, we obtain that $$T_{E_1}|_{C_1}=T_{D_1}|_{C_1}, \ \forall \ D_1\in |L_1| \  \textup{general.}$$

All general members $D_1\in |L_1|$ are therefore tangent along $C_1$. 
 
 \medskip
 
 Let $\mu_2:X_2\rightarrow X_1$ be the blow up of $X_1$ along $C_1$, let $E_2$ be its exceptional divisor and set $\mu=\mu_1\circ\mu_2$. We obtain the following coefficients:
\begin{center}
$\begin{array}{c} 
K_{X_2}=\mu^*K_X+3E_1'+5E_2\\
\mu^*D=D_2+2E_1'+3E_2,
\end{array}$
\end{center}
where $D_2$ and $E_1'$ are the strict transforms of $D_1$ and $E_1$ respectively.
Take two distinct general divisors $D_1', D_1''\in |L_1|$. As they are tangent at the generic point of $C_1$, its blow-up $\mu_2$ will not separate their strict transforms  $D_2'$ and $D_2''$. Explicitly, on $X_2$ we have the surfaces \[Q':=D_2'|_{E_2}\simeq \PP(\mathcal{N}^*_{C_1/D_2'}) \  \textup{and} \ Q'':=D_2''|_{E_2}\simeq \PP(\mathcal{N}^*_{C_1/D_2''}),\] which coincide as the normal sheaves are the same by the commutative diagram:
\begin{center}
$\begin{array}[c]{ccccccccc}
  0 &  \rightarrow &  \mathcal{T}_{C_1}  & \rightarrow &  \mathcal{T}_{D_2'}|_{C_1} & \rightarrow & \mathcal{N}_{C_1/{D_2'}} & \rightarrow & 0\\
  && \rotatebox{90}{=} && \rotatebox{90}{=} \\
 0 &  \rightarrow &  \mathcal{T}_{C_1}  & \rightarrow &  \mathcal{T}_{D_2''}|_{C_1} & \rightarrow & \mathcal{N}_{C_1/{D_2''}} & \rightarrow & 0.
\end{array}$
\end{center}
Thus the two strict transforms intersect along a surface $S\subset X_2$. Similarly, using that $T_{E_1}|_{C_1}=T_{D_1'}|_{C_1},$ we have that $S$ is also contained in $E_1'$. If $\mu_3:X_3\to X_2$ is the blow-up of $X_2$ along $S$, from the following computations:  
\begin{center}
$\begin{array}{c}
\mu_3^*E_1=E_1'+E_3, \\
 \mu_3^*E_2=E_2'+E_3, \\
 \mu_3^*D_2=D_3+E_3,
 \end{array}$
\end{center}
 we obtain $a_3=9$ and $r_3=6$, where $D_3$ is the strict transform of $D_2$ through $\mu_3$. This again contradicts condition \eqref{airi}.
 
\medskip
 
\noindent\textbf{Rank three:} By Lemma \ref{Morse}, in this case the polynomial $g$ only depends on the variable $x_4$ and modulo a change of coordinates the equation of a general anticanonical member $D\in |-K_X|$ is:
$$x_1^2+x_2^2+x_3^2+ x_4^k=0,$$ where $k\geq 3$. The strict transform $D_1$ of such a divisor has a unique singularity of multiplicity two at the point $[0:0:0:1]$. The rank of its hessian remains at least equal to three, while the compound term becomes of degree $k-2$. As before, the singularities of $D_1$ are contained in the curve $C:=\mathrm{Bs}|L_1|\cap E_1$.

We show that in fact every general $D_1$ is smooth, allowing us to conclude that $k=3$. This is essentially done by contradiction and using the same sequence of blow-ups as in rank two, only this time we need to be more precise in order to obtain the tangency condition.

\emph{Step 1: Assume $k>3$ and thus each general $D_1\in |L_1|$ is singular. Then these singularities are contained in the same irreducible component of $C$.}  As in the previous case, this is immediate since we have a finite number of irreducible components and an infinite number of divisors. Denote this component by $C_1$.

In what follows, let $|L_1|^0$ be the Zariski open set in $|L_1|$ such that for all $D_1\in |L_1|^0$ we have $D_{1,sing}\subset \mathrm{Bs}|L_1|$. Consider $$\mathcal{U}:=\{(D_1,x_1)\ \vert \ D_1\in |L_1|^0, \ x_1\in D_{1,sing}\cap C_1\}\subset |L_1|^0\times X_1$$ to be the universal family over $|L_1|^0$ and denote by $p_1$ and $p_2$ the projections on its two components.

\emph{Step 2: Two general members of $|L_1|^0$ are not singular at the same point of the curve $C_1$.} Choose the first general member $D'_1\in |L_1|^0$ and let $x_1\in C_1$ be its singular point. The set $$F_1:=\left\{D_1\in |L_1|^0 \vert \ x_1\in D_{1,sing}\right\}$$ is a fiber of $p_2$, so it is closed. First observe that this set cannot be dense. Otherwise, every general element $D_1\in |L_1|^0$ would be singular at $c_1$, say of multiplicity $m\geq 2$. Then by blowing up $X_1$ at $c_1$ one would obtain a contradiction to condition \eqref{airi}, since $a_2=6$ and $r_2=2+m\geq 4$. Hence $F_1$ is closed, but not dense, and by generality we can choose $D''_1$ in its complement, this way making sure that its singular point does not coincide with $c_1$.

\emph{Step 3: Having fixed $D_1'$ and $D''_1$ as above, let $P:=\langle D_1',D_1''\rangle$ to be the pencil that they generate. Then the singular loci of the members of $P$ cover the entire curve $C_1$.} The set $$\mathcal{T}=\{(D_1,x_1)\ |\ D_1\in P, \ x_1\in D_{1,sing}\cap E_1\}$$ is a closed subset of $\mathcal{U}$, and by Step 1 the second projection $p_2:\mathcal{U}\rightarrow X_1$ maps it onto a closed subset of $C_1$. Using Step 2 and the continuity of $p_2$, we conclude that its image must be all of $C_1$. 

\emph{Step 4: We obtain a contradiction and conclude that $k=3$.}

By Remark \ref{post-tangency} we obtain that $T_{D_1'}|_{C_1{_{gen}}}=T_{D_1''}|_{C_1{_{gen}}}$. Similarly to the rank two case, we blow up $X_1$ along $C_1$ and because we have the following exact sequences: 
\begin{center}
$\begin{array}[c]{ccccccccc}
  0 &  \rightarrow &  \mathcal{T}_{C_1{_{gen}}}  & \rightarrow &  \mathcal{T}_{D_1'}|_{C_1{_{gen}}} & \rightarrow & \mathcal{N}_{C_1/{D_1'}} & \rightarrow & 0\\
  && \rotatebox{90}{=} && \rotatebox{90}{=} \\
 0 &  \rightarrow &  \mathcal{T}_{C_{1_{gen}}}  & \rightarrow &  \mathcal{T}_{D_1''}|_{C_1{_{gen}}} & \rightarrow & \mathcal{N}_{C_1/{D_1''}} & \rightarrow & 0,
\end{array}$
\end{center}
we deduce the fact that the strict transforms $D_2'$ and $D''_2$ intersect along a surface $S$. However, contrary to the rank two case, here the strict transform of $E_1$ does not contain $S$ since the tangent cone is smooth at the generic point of $C_1$.

By doing the same blow-up of $X_2$ along $S$ and using that $S\not\subset E_1$ we again arrive at a contradiction of condition \eqref{airi} as the coefficients are $a_3=6$ and $r_3=4$.

\medskip

\noindent\textbf{Rank four:} This case does occur, and it is precisely the one illustrated in the example \cite[Ex. 2.12]{HV11} mentioned in the introduction. Together with the only possible case in rank three, this proves the theorem. \end{proof}

Note that throughout the proof, despite having started with a fixed singularity of $|-K_X|$, we have come across singularities of elements in $|L_1|$ "moving" along a curve $C_1$. We now see what happens if that had already been the case for $|-K_X|$.

\subsection{Moving singularities}

As usual, $X$ is a four-dimensional Fano manifold. Proposition \ref{known},\ref{canonical} states that all general elephants have isolated singularities. Consider the set \[V:=\{x\in D_{sing}\cap \mathrm{Bs}|-K_X| \, | \, D\in |-K_X| \ \textup{general}\}.\] If it contains a component of strictly positive dimension, we say that $|-K_X|$ has moving singularities along the component in question. As all our computations are local, we analyze the two possible dimensions separately.

\subsubsection{The curve case}

We are in the most elementary situation of a moving singularity: in the base locus of $|-K_X|$ there exists a curve, which we denote by $C_0$, such that for all general $D\in |-K_X|$ there exists a point $x\in D_{sing}\cap C_0$. Suppose that the set of such points is dense in $C_0$. Here is the central result of this section:

\begin{theorem}
\label{mobcv}
Let $X$ be a four-dimensional Fano manifold and using the terminology above suppose that $|-K_X|$ has moving singularities along a curve $C_0$. Then around this point each general elephant is defined by an equation of the form:\[x_1^2+x_2^2+x_3^2+x_4^2=0.\]
\end{theorem}

Just as we have done previously, a good approach is to start by blowing up the singular locus, the difference being that instead of focusing on a single point and a specific divisor, we must consider the linear system as a whole.

\medskip

\noindent\textbf{Global geometric context:} As a general elephant $D$ moves through the anticanonical system, its isolated singularities describe a curve $C_0$. By Remark \ref{post-tangency} we have that $T_D|_{{C_0}_{gen}}$ is independent of $D\in |-K_X|$. This means that by blowing up $C_0$ all strict transforms of general elephants will have a surface in common. Denote this blow-up by $\mu_1:X_1\rightarrow X$, the exceptional divisor by $E_1$ and the common surface by $S_1$. By Bertini's Theorem, if a general divisor $D$ is singular at a point $x\in C$, then \[D_{1,sing}\cap \PP^2\subset D_{1,sing}\cap \PP^2\cap S_1,\] where $D_1$ is the strict transform of $D$ and $\PP^2$ is the projective plane in $E_1$ mapping onto $x$.

Note that after this step the coefficients in Notation \ref{notation} are $a_1=2$ and $r_1=1$. 

\medskip

\noindent\textbf{Local coordinates:} The purpose of this discussion is to find a way to check if the strict transform of a general elephant is singular, and then to analyze its singularities. 

For simplicity, start by choosing coordinates on an open set $U$ such that $C_0$ is given by: \begin{align*}C_0 \: : \:\{ x_2=x_3=x_4=0 \}\end{align*} and single out two general elephants: one denoted by $D$, which is singular at the origin and the second, denoted by $\widetilde{D}$, which is smooth inside $U$. By restricting $U$ we may assume that the origin is the only singular point of $D$ and, up to a coordinate change, that the local equation of $\widetilde{D}$ is precisely $$x_3=0.$$ This choice of coordinates, convenient for obtaining a straightforward expression of $\mu_1$, comes at the expense of the precise form of $D$ given by the Morse Lemma in \hbox{Section \ref{fixedsec}.} 

Denote by $\mu_1:U_1\rightarrow U$ the blow-up of $U$ along $C_0$. We choose the chart on $U_1$ given by  \begin{align*}\left\{ (x_1,x_2,x_3,x_4)(z_2,z_3) \:\left \vert \left. \begin{array}{c} x_3=x_2z_2 \\ x_4=x_2z_3 \end{array}  \right. \right. \right\}\end{align*}  and the following local coordinates on it: \begin{align*}(u_1,u_2,u_3,u_4) \longrightarrow (u_1,u_4,u_2u_4,u_3u_4)(u_2,u_3).\end{align*} Note that in this chart $\mu_1$ is given by $$ (u_1,u_2,u_3,u_4)\rightarrow (u_1,u_4,u_2u_4,u_3u_4), $$ the exceptional divisor $E_1$ has the equation $u_4=0$ and the projective plane above the origin has the equations $u_1=u_4=0$. 

\medskip

The surface $D_1\cap E_1$ has two irreducible components: $S_1$, the surface that is common to all general elements in $|L_1|$, and the $\PP^2$ above the origin. By intersecting them, we obtain a curve, denoted by $C$, which will contain the singularities of $D_1$. Since the singular point of $\widetilde{D}$ is outside of $U$ we have that $S_1=\widetilde{D}_1\cap E_1$. It is then easy to deduce that the local expressions of $S_1$ and $C$ are \begin{align*}S_1=\{u_4=u_2=0\} \ \mbox{and} \ C=\{u_1=u_4=u_2=0\}.\end{align*} 

\medskip

If $f$ is the local equation of $D$, write  $f=q+h$ where $q$ is the quadratic part of $f$ and $h$ contains higher order terms. Denote by $M_D=(m_{ij})_{1\leq i,j\leq 4}$ the matrix of $q$ viewed as a quadratic form. This is a symmetric matrix whose double is equal to the hessian matrix of $f$ at the origin.

As the tangent spaces of $D$ and $\widetilde{D}$ coincide at every point of $C_{0,gen}$, the jacobian of $f$ is proportional to the vector $(0\  0\ 1 \ 0)$ at each point $x\in C_{0,gen}$. We deduce that $f$ does not contain monomials of type $x_1^kx_2$ or $x_1^kx_4$, in particular $M_D$ is of the following form: \begin{align*}M_{D}=\begin{pmatrix}0 &0 & m_{13}& 0\\ 0 &m_{22}& m_{23} &m_{24} \\ m_{13}& m_{23}& m_{33}&m_{34} \\0& m_{24}& m_{34} &m_{44} \end{pmatrix}. \end{align*} This matrix will be the main object of study in our case-by-case analysis.

\begin{proof}[Proof of Theorem \ref{mobcv}]
We proceed to eliminate all cases in which the rank of $M_D$ is less than four.

\paragraph{Case 1:} $m_{13}\neq 0.$

When restricting $J_{D_1}$ to $C$ we obtain:
\begin{align*}
 J_{D_1}|_{C}= \begin{pmatrix} 0 \\ 0 \\ 0 \\  m_{44}u_3^2+2m_{24}u_3+m_{22}\end{pmatrix},
\end{align*}
thus the singular points of $D_1$ are given by:
\begin{align} \label{D1sing} \left \{\begin{array}{l} u_1=u_2=u_4=0 \\ [6pt] 
m_{44}u_3^2+2m_{24}u_3+m_{22}=0
\end{array} \right. \end{align}

We separately analyze the cases corresponding to different ranks of $M_{D}$. The aim is to prove that rank four is the only possibility. 

\medskip

\noindent\textbf{Rank one:}
This case doesn't occur because by Theorem \cite[Thm. 4.4]{Kollar2} each moving singularity is of type $cA$, meaning that its degree two part is at least of rank two as a quadratic form.

\medskip

\noindent\textbf{Rank two:}
First we show that $D_1$ is singular at $C_{gen}$, i.e. that all the coefficients of the degree two equation in (\ref{D1sing}) are zero. Denote by $M_{ij}$ the $3\times 3$ minors in $M_D$ obtained by eliminating row $i$ and column $j$. Since $M_D$ has rank two, all of the $M_{ij}$ are zero. Note that $$M_{44}=m_{13}^2m_{12},\  \ M_{24}=m_{13}^2m_{24} \mbox{ and } M_{22}=m_{13}^2m_{44},$$ and since $m_{13}\neq 0$ the result is immediate.

By Remark \ref{post-tangency}, this implies that all elements in $D_1\in |L_1|$ are tangent along $S_1$. If we blow-up $U_1$ along this surface, the strict transforms of these divisors will have a surface in common, that is in its generic point defined by $\PP(\mathcal{N}^*_{S_1 / D_1})$. Denote it by $S_2$. We construct a series of blow-ups as follows:

\[
  \xymatrix@!{
              & {U} & \ar[l]_{\mu_1}   {U_1} & \ar[l]_{\mu_2}   {U_2} &  \ar[l]_{\mu_3}  {U_3,} 
                    }
\]
where $\mu_2$ and $\mu_3$ are the blow-ups of $U_1$ along $S_1$ and of $U_2$ along $S_2$ respectively. By computing the coefficients
\begin{align*} \left \{ \begin{array}{c} a_2=a_1+1=2+1=3 \\ r_2=r_1+1=1+1=2 \end{array} \right. \ \text{and} \ \left \{ \begin{array}{c} a_3=a_2+1=4 \\ r_3=r_2+1=3\end{array} \right. \end{align*} we obtain an immediate contradiction to condition \eqref{airi}.

\medskip

\noindent\textbf{Rank three:} 
All of the coefficients of the equation $m_{44}u_3^2+2m_{24}u_3+m_{22}=0$ being equal to zero would imply $\mathrm{rank}(M_D)=2$. Since $$det(M_D)=m_{13}^2 \times (m_{22}m_{44}-m_{24}^2)=0,$$ we may assume, up to choosing a different chart on $U_1$, that $m_{44}\neq 0$. Since $m_{22}m_{44}-m_{24}^2=0$, we have a degree two equation with the double root $u_3=-\dfrac{m_{24}}{m_{44}}$, namely $$D_{1,sing}=  \left(0,\,0,\,-\dfrac{m_{24}}{m_{44}}, \, 0\right).$$

Thus if $M_D$ is of rank three, the strict transform $D_1$ has exactly one singular point, the same being true for all general elephants. This means that the linear system of strict transforms $|L_1|$ also has singularities that are moving inside of $S_1$. In order to show that these singularities will not cover the entire surface it suffices to prove that $\mathrm{Bs}|L_1|$ has a reduced structure at $S_1$ (see discussion in Section \ref{surface case} for details). This becomes apparent when we look at the local equations:  since as before $m_{44}$, $m_{24}$ and $m_{22}$ cannot be all at once equal to zero, the equation of $D_1\cap \widetilde{D}_1=D_1|_{\{u_2=0\}}$ contains at least one of the monomials $u_3u_4$, $u_3^2u_4$ or $u_4$, thus this intersection only contains $S_1=\{u_2=u_4=0\}$ with multiplicity one.

The singularities of general divisors in $|L_1|$ must then move along a curve $C_1\subset S_1$. By Remark \ref{post-tangency}, all elements of $|L_1|$ must be tangent along $C_1$. We want to explicitly construct the blow-up of $U_1$ along $C_1$ and at the same time keep track of the singular points of $D_1$.

\medskip

We need to control the rank of the singularity of $D_1$ in order to understand whether we have improved the initial situation by doing the first blow-up. We do a coordinate change that brings the singular point onto the origin:
$$(u_1,\,u_2,\, u_3, \,u_4) \mapsto (u_1,\, u_2, \, u_3+\dfrac{m_{24}}{m_{44}},\, u_4)$$ and we denote by $M_{D_1}=(p_{ij})_{1\leq i,j,\leq 4}$ the matrix of its degree two part as a quadratic form. Again, if $f_1$ is a local equation of $D_1$, let $f_1=q_1+h_1$, where $q_1$ is the quadratic part of $f_1$.

We claim that $p_{12}=m_{13}$ and $M_{D_1}$ is of the form:  \begin{align}\label{MD1} M_{D_1}=\begin{pmatrix}0 &p_{12} & 0& p_{14}\\ p_{12} &0& 0 &p_{24} \\ 0& 0& 0&0 \\p_{14}& p_{24}& 0 &p_{44} \end{pmatrix}.\end{align}

Indeed, take an arbitrary monomial of $f$ and trace it throughout the first blow-up and the coordinate change:  
\begin{align*} \textswab{m}=x_1^{d_1}x_2^{d_2}x_3^{d_3}x_4^{d_4} \longrightarrow \overline{\textswab{m}}=u_1^{d_1}u_2^{d_3}u_3^{d_4}u_4^{d_2+d_3+d_4-1}Ê& \longrightarrow \\ 
\longrightarrow  \textswab{m}_1=u_1^{d_1}u_2^{d_3}\left(u_3-\dfrac{m_{24}}{m_{44}}\right)^{d_4}u_4^{d_2+d_3+d_4-1}\end{align*}Ê
Note that $\textswab{m}_1$ splits into $d_4+1$ monomials of degrees $d_1+d_2+2d_3+d_4+k-1$, where $k\in\{0\ldots d_4\}$.

\medskip

\textit{Step 1: } $p_{22}=p_{23}=p_{33}=0$.  A monomial in $\textswab{m}_1$ contributing to $M_{D_1}$ will be of degree two and it follows from the expression above that there will be none of type $u_2^2$, $u_2u_3$ or $u_3^2$. Indeed, for all of these monomials we have $d_3+k=2$, automatically increasing the total degree to at least three.

\medskip

\textit{Step 2: } $p_{13}=0$. The monomial $u_1u_3$ can only be obtained if $d_1=1$ and $k=1$, while all other powers are zero. This implies $d_4=1$ and $d_3=d_2=0$, thus the coefficient of $u_1u_3$ is $2m_{14}=0$.
\medskip

\textit{Step 3: } $p_{12}=m_{13}$ As in the previous step, we obtain that the coefficient of $u_1u_2$ is $2m_{13}$.

\medskip

\textit{Step 4: } $p_{34}=0$.  The monomial $u_3u_4$ can be obtained from either $x_2 x_4$ or $x_4^2$ in $f$, corresponding to $k=1$, $d_4=2$ and $k=1$, $d_2=d_4=1$ respectively. Its coefficient $2p_{34}$ will be $$m_{44}\times 2\left(-\dfrac{m_{24}}{m_{44}}\right)+2m_{24}=0.$$

\medskip

\textit{Step 5: } $p_{11}=0$. The contributions to $p_{11}$ come from monomials of the form $x_1^2x_2$ and $x_1^2x_4$. Because of the tangency condition relating the jacobian of $D$ with that of $\widetilde{D}$, the coefficients in $f$ of these two monomials are zero.

\medskip

This proves the claim. We have thus obtained a much simpler matrix $M_{D_1}$ after the first blow-up, though its rank can still be equal to three.

\medskip

We now construct the second blow-up. Since we cannot change coordinates while maintaining the format of $M_{D_1}$, we will consider an arbitrary curve in $S_1$ passing through the origin and denote it by $C_1$. This curve is smooth at the origin because of the general choice of $D$. Locally it is given by: \begin{align*} \left \{\begin{array}{l} u_2=u_4=0 \\ [6pt] 
g(u_1,u_3)=0
\end{array}, \right. \end{align*} where $g$ is an arbitrary holomorphic function such that $g(0,0)=0$. 
\medskip

Denote by $\mu_2:U_2\rightarrow U_1$ the blow-up of $U_1$ along $C_1$. One of the charts on $U_2$ is given by  \begin{align*}\left\{ (u_1,u_2,u_3,u_4)(t_1,t_2) \: \left \vert \left\{ \begin{array}{c} u_2=t_1g(u_1,u_3) \\ u_4=t_2 g(u_1,u_3) \end{array} \right. \right. \right\}\end{align*}  and we choose the following local coordinates on it: $v_1:=t_1,\, v_2:=t_2, \,v_3:=u_1$ and $ v_4:=u_3$. The chart becomes: \begin{align*}(v_1,v_2,v_3,v_4) \longrightarrow (v_4,v_1g(v_3,v_4),v_3,v_2g(v_3,v_4))(v_1,v_2).\end{align*} In this chart the exceptional divisor $E_2$ is given by $g(v_3,v_4)=0$ and the projective plane above the origin has the equations $v_3=v_4=0$. The strict transform of $E_1$, denoted by $E_1'$, is given by $v_2=0$.

\medskip 

Denote by $D_2$ the strict transform of $D_1$ through $\mu_2$. We will show that $D_2$ is singular at the origin and we will compute the rank of $M_{D_2}$ at this point.

Consider an arbitrary monomial in $f_1$, denoted by $\textswab{m}_1=u_1^{d_1}u_2^{d_2}u_3^{d_3}u_4^{d_4}.$ Note that since $S_1\subset D_1$ we must have $d_2+d_4>0$. As before, its contribution to the local equation of $D_2$ is $\textswab{m}_2=v_1^{d_2}v_2^{d_4}v_3^{d_3}v_4^{d_1}g(v_3,v_4)^{d_2+d_4-1}.$ A short computation shows that the partial derivatives of $\textswab{m}_2$ vanish at the origin.

If the origin is not an isolated singular point, we are done. Indeed, if $D_2$ is singular along an entire curve, by Remark \ref{post-tangency} all elements in $|L_2|$ are tangent along a surface denoted by $S_2$. We construct the following blow-up sequence:
\[
  \xymatrix@!{
              & {U} & \ar[l]_{\mu_1}   {U_1} & \ar[l]_{\mu_2}   {U_2} &  \ar[l]_{\mu_3}  {U_3}  &  \ar[l]_{\mu_4}  {U_4,}  
                    }
                    \]
where $U_3=\mathrm{Bl}_{S_2}U_2$ and $U_4=\mathrm{Bl}_{S_3}U_3$, where $S_3=\PP(\mathcal{N}^*_{S_2 / D_2})$ is the surface that the strict transforms of elements in $|L_2|$ have in common. The coefficients are:
\begin{align*} \left \{ \begin{array}{c} a_2=a_1+2=4 \\ r_2=r_1+1=2 \end{array} \right.\text{,} \ \left \{ \begin{array}{c} a_3=a_2+1=5 \\ r_3=r_2+1=3\end{array} \right.  \text{and} \  \left\{ \begin{array}{c} a_4=a_3+1=6 \\ r_4=r_3+1=4 \end{array}, \right.\end{align*} a contradiction to condition \eqref{airi}. 

Assume now that the origin is an isolated singular point of $D_2$. Then it is neither a fixed singularity for the system $|L_2|$ nor is it a moving singularity along a surface. Indeed, a short computation shows that $\widetilde{D}_2$ is smooth at the origin and Proposition \ref{surface} allows us to eliminate the surface case if $\mathrm{Bs}|L_2|$ has a reduced structure at $S_{2,gen}$. If this is not the case, all elements in $|L_2|$ are tangent along $S_2$ and we repeat the sequence of blow-ups above in order to derive a contradiction.

We now proceed to determining the rank of this singularity. If a monomial in $\textswab{m}_2$ is of degree two then $d_2+d_4=1$ and $d_1+d_3=1$, in particular it comes from certain degree two monomials in $f_1$. Using that $M_{D_1}$ is of the form in \eqref{MD1}, we obtain that $M_{D_2}$ is of the following form: \begin{align*} M_{D_2}=\begin{pmatrix}0 & 0 & 0& p_{12}\\ 0 &0& 0 &p_{14} \\ 0& 0& 0&0 \\p_{12}& p_{14}& 0 &0 \end{pmatrix}.\end{align*} 

This is a rank two matrix, therefore we can apply the same strategy as in the previous case. The coefficients will not be exactly the same since we have already blown up two subvarieties, we will however obtain the same type of contradiction.

\medskip

For this we need to construct a sequence of five blow-ups:
\[
  \xymatrix@!{
              & {U} & \ar[l]_{\mu_1}   {U_1} & \ar[l]_{\mu_2}   {U_2} &  \ar[l]_{\mu_3}  {U_3}  &  \ar[l]_{\mu_4}  {U_4}  &  \ar[l]_{\mu_5}  {U_5,}  
                    }
                    \]
where as before $U_1=Bl_{C_0}U$ and $U_2= Bl_{C_1}U_1$. The morphisms $\mu_3$, $\mu_4$ and $\mu_5$ will be the blow-ups along a curve $C_2$ and two surfaces denoted by $S_3$ and $S_4$ respectively. 

The choices of the centers are straightforward, we proceed exactly as in the previous rank two case: \begin{enumerate}[$\cdot$] \item$C_2\subseteq U_2$ is the curve along which the singularities of general members of $|L_2|$ move, \item $S_3\subseteq U_3$ is the surface along which general members of $|L_3|$ are tangent, \item $S_4\subseteq U_4$ is the surface in $\mathrm{Bs}|L_4|$ that exists because of this tangency. \end{enumerate}The only thing we need to additionally keep track of is the interaction between the exceptional loci. 

We briefly come back to the local picture in order to eventually describe $\mu_4$. The first observation is that the origin belongs to both divisors $E_2$ and $E_1'$, the strict transform of $E_1$ through $\mu_2$. As the sequence of blow-ups does not depend on the coordinate choice, we have just used a local computation to show that the singular point of the strict transform of a general $D\in |-K_X|$ belongs to $E_1'\cap E_2$. The same must be then true for the curve formed precisely by these singular points, that is to say $C_2$.

We claim that the only exceptional divisor that $S_3$ belongs to is $E_3$. Indeed, a short computation shows that \[ T_{E_2}|_{C_{2,gen}}\neq T_{\widetilde{D}_2}|_{C_{2,gen}}  \textup{ and } T_{E_1'}|_{C_{2,gen}}\neq T_{\widetilde{D}_2}|_{C_{2,gen}}\] where $\widetilde{D}_2$ is the strict transform of $\widetilde{D}$ through $\mu_1\circ \mu_2$. As $T_{\widetilde{D}_2}|_{C_{2,gen}}=T_{D_2}|_{C_{2,gen}}$ by Remark \ref{post-tangency}, this proves that the divisor $E_4$ will be disjoint from all strict transforms of $E_1$ and $E_2$, but will have a surface in common with $E_3$. At this stage, the coefficients are: 
\begin{align*} \left \{ \begin{array}{c} a_2=a_1+2=4 \\ r_2=r_1+1=2 \end{array} \right. \ \text{,} \ \left \{ \begin{array}{c} a_3=a_1+a_2+2=8 \\ r_3=r_1+r_2+1=4\end{array} \right.  \text{and} \  \left\{ \begin{array}{c} a_4=a_3+1=9 \\ r_4=r_3+1=5 \end{array}, \right.\end{align*} which doesn't yet allow us to conclude. The fifth blow-up is of the surface $S_4\subset E_4$ defined above, which may or may not also be included in $E_3$. The two cases both lead us to a contradiction to condition \eqref{airi}:
\begin{align*} \left \{ \begin{array}{c} a_5=a_4+a_3+1=18 \\ r_5=r_4+r_3+1=10\end{array} \right.  \text{or} \  \left\{ \begin{array}{c} a_5=a_4+1=10 \\ r_5=r_4+1=6 \end{array}. \right.\end{align*} 

\medskip

\paragraph{Case 2:} $m_{13}=0$.

\medskip

This is a degeneration of the previous situation and as such will be easier to exclude. Geometrically, the condition says that the intersection $D_1\cap E_1$ has a non-reduced structure along the $\PP^2$ above the origin. As before, we have that  \begin{align*}S_1=\{u_4=u_2=0\} \ \mbox{and} \ C=\{u_1=u_4=u_2=0\}.\end{align*} The problem here is that $D_1$ may be singular outside $C$: if we restrict the jacobian of $D_1$ just to the $\PP^2$ above the origin, we obtain
 \[J_{D_1}|_{\PP^2} = \begin{pmatrix} 0 \\ 0 \\ 0 \\  m_{22}+ 2m_{23}u_2+m_{33}u_2^2+2m_{34}u_2u_3+2m_{24}u_3+m_{44}u_3^2\end{pmatrix} \footnote{this is not the case if $m_{13}\neq 0$, as the first line of $J_{D_1}|_{\PP^2}$ is $2m_{13}u_2.$}
\]
which is included in $C$ iff either $m_{23}\neq 0$ while $m_{22}=m_{33}=m_{34}=m_{24}=m_{44}=0$ or $m_{33}\neq 0$ and $m_{22}=m_{23}=m_{34}=m_{24}=m_{44}=0$, both cases leading to $D_1$ being singular along the entire curve $C$. The latter is impossible since it would mean that the rank of $M_D$ is one and \cite[Thm. 4.4]{Kollar2} implies it should be at least equal to two. The former also leads to a contradiction: we have that $M_D$ is of rank two and $D_1$ is singular along a curve. We perform the same blow-ups as in the rank two case and obtain the coefficients $a_3=4$ and $r_3=3$ which contradict condition \eqref{airi}. \end{proof}

\subsubsection{The surface case}
\label{surface case}
Now that we have discussed the situation where the singularities of general elements in $|-K_X|$ move along a curve $C$, we claim that this is the maximal-dimensional case that we need to consider. Specifically, in what follows we show that the set $$\{x\in D_{sing} | D\in |-K_X|^0 \}$$ cannot contain a surface: we prove that if $\mathrm{Bs}|-K_X|$ contains a reduced surface $S$, then in fact the singular points of all general $D\in |-K_X|^0$ belong to a curve included in $S$.

We are only concerned with the smooth points $x\in S$, since the singular ones already belong to a subset of the desired codimension. Fix a general element $D\in |-K_X|^0$ that is singular at $x$. The divisor $S\subset D$ is not Cartier at $x$, as otherwise all points in $D_{sing}\cap S$ would be singular points of $S$. Since every $M\in |-K_X|_D$ is Cartier at $x$ we see that there exists another component $R\subset M$ with $x\in R$. We can then decompose $|-K_X|_D$ as follows:
\begin{equation} \label{decomp}|-K_X|_D=S+|R_D|,\end{equation} such that $x\in \mathrm{Bs}|R_D|$. Note that the linear system $|R_D|$ may have fixed components, as $\mathrm{Bs}|-K_X|$ possibly contains other surfaces aside from $S$.

We show that $\mathrm{Bs}|R_D|$ is independent of the initial choice of $D$. Indeed, fix $D'\in |-K_X|$ to be another general element and through the same process construct $|R_{D'}|$. Then $D\cap D'$ is a subscheme of $D$ having $S$ as an irreducible component. Denote by $T:=\mathrm{Supp}(D\cap D'\setminus S)$. We have that $||-K_X|_D|_T=||-K_X|_{D'}|_T=|-K_X|_T$ since the following diagram of restrictions is commutative:
\[
  \xymatrix{
  		 &  H^0(D,\oo_D(-K_X)) \ar[rd] \\
                H^0(X,\oo_X(-K_X)) \ar[ru] \ar[rd]  & & H^0(T,\oo_T(-K_X))  \\
                & H^0(D', \oo_{D'}(-K_X)) \ar[ru]                   
               }
                    \]
                    
By restricting \eqref{decomp} and the analogous decomposition of $|-K_X|_{D'}$ with respect to $|R_{D'}|$ to $T$ we obtain that \[|-K_X|_T=S_T+|R_D|_T=S_T+|R_{D'}|_T.\] So $|R_D|_T=|R_{D'}|_T$, which means $\mathrm{Bs}|R_D|=\mathrm{Bs}|R_{D'}|$ since both base loci are included in $T$.

\begin{proposition}
\label{surface}
Let $|L|$ be a linear system without fixed components on a projective four-dimensional manifold $X$. If $S\subset \mathrm{Bs}|L|$ is a reduced surface and all general elements in $|L|$ are smooth in codimension one, then the set $W=\{x \in D_{sing} \,| \, D\in |L|^0\}\cap S$ is at most of dimension one. 
\end{proposition}

\begin{proof}
Since every $x\in W$ is either a singular point of $S$ or belongs to $\mathrm{Bs}|R_D|\cap S$ by the previous discussion, it is enough to show that $S\cap \mathrm{Bs}|R_D|\subsetneq S$. Indeed, suppose $S\subset \mathrm{Bs}|R_D|$, then we can write $$|L|_D=2S+|R_D'|,$$ meaning that for every $D'\in|L|$ we have that $D'\cap D$ contains $2S$. This is equivalent to $\mathrm{Bs}|L|$ having a non-reduced structure at $S_{gen}$, which contradicts the hypothesis. \end{proof}

\begin{proof}[Proof of Theorem \ref{thm 2}] We now check that the hypotheses of Proposition \ref{surface} are true for a four-dimensional Fano manifold $X$ and the anti-canonical system $|-K_X|$. Suppose for a contradiction that $\mathrm{Bs}|-K_X|$ has a non-reduced structure along an irreducible surface $S$. By \cite[Prop 4.2]{Kawamata1}, if we consider two general elephants $D$ and $D'$, they give rise to an lc pair $(D, D\cap D')$. But $S$ is a component of $D\cap D'$ of at least multiplicity two, a contradiction.

We then apply Proposition \ref{surface} and conclude that the anticanonical system $|-K_X|$ either has fixed singularities or singularities moving along a curve.  By Theorem \ref{fixed} and Theorem \ref{mobcv}, we obtain that locally analytically these points are of the form: \[x_1^2+x_2^2+x_3^2+x_4^2=0 \text{ or } x_1^2+x_2^2+x_3^2+x_4^3=0.\] \end{proof}


\begin{thebibliography}{1}

\bibitem{Ambro} F. Ambro, {\em Ladders on Fano varieties}, J. Math. Sci. (New York) 94 (1999), no. 1, pp. 1126-1135.

\bibitem{AGV} V. I. Arnold, S. M. Gusein-Zade, A. N. Varchenko, {\em Singularities of Differentiable Maps Vol I}, Birkh\"auser (1985).

\bibitem{Debarre} O. Debarre, {\em Higher-Dimensional Algebraic Geometry,} Universitext, Springer, New York (2001).

\bibitem{Floris} E. Floris, {\em Fundamental Divisors on Fano Varieties of index n-3}, Geom. Dedicata 162 (2013), pp. 1-7.

\bibitem{Fundamental} O. Fujino, {\em Fundamental theorems for the log minimal model program},  Publ. Res. Inst. Math. Sci. 47 (2011), no. 3, pp. 727-789.

\bibitem{FuGo} O. Fujino and Y. Gongyo,  {\em On canonical bundle formulae and subadjunctions}, Michigan Math. J. 61 (2012), no. 2, pp. 255-264.

\bibitem{Hartshorne} R. Hartshorne, {\em Algebraic Geometry}, Graduate Texts in Mathematics, Springer, Heidelberg (1977).
  
  \bibitem{HV11} A. H\"oring, C. Voisin, {\em Anticanonical divisors and curve classes on Fano manifolds,} Pure and Applied Mathematics Quarterly 7 (2011), No. 4, pp. 1371-1393, Special volume dedicated to Eckart Viehweg.
  
  \bibitem{Kawamata1} Y. Kawamata, {\em On Effective non-vanishing and base-point-freeness}, Asian J. of Math. \textbf{4} (2000), no. 1, pp. 173-181. Kodaira's issue.
  
  \bibitem{Kawamata2} Y. Kawamata {\em Subadjunction of log canonical divisors II}, Amer. J. Math. 120 (1998), no. 5, pp. 893-899.
  
  \bibitem{KMM} Y. Kawamata, K. Matsuda and K. Matsuki, {\em Introduction to the Minimal Model Problem}, Advanced Studies in Pure Mathematics 10, 1987, Algebraic Geometry, Sendai (1985), pp. 283-360. 

 \bibitem{Kollar}  J. Kollar, {\em Kodaira's canonical bundle formula and subadjunction}, Oxford Lecture Series in
Mathematics and its Applications 35 (1985), chapter 8, pp. 121-146. 

\bibitem{Kollar3} J. Kollar, {\em Flips, flops, minimal models, etc.} Surveys in differential geometry (Cambridge, MA, 1990), pp. 113-199, Lehigh Univ., Bethlehem, PA, 1991.

\bibitem{Kollar2} J. Kollar, {\em Singularities of Pairs}, Algebraic Geometry, Santa Cruz (1995), Proc. Symp. Pure Math. 62 (1997), Amer. Math. Soc., Providence, RI, pp. 221-287. 

\bibitem{Laz} R. Lazarsfeld, {\em Positivity in Algebraic Geometry. II.} Positivity for vector bundles and multiplier ideals,  Ergebnisse der Mathematik und ihrer Grenzgebiete. 3. Folge. A Series of Modern Surveys in Mathematics, Springer-Verlag, Berlin (2004).

\bibitem{Mor} S. Mori, {\em On $3$-dimensional terminal singularities}, Nagoya Math. J. 98 (1985), pp. 43-66.

 \bibitem{MorMuk} S. Mori and S. Mukai, {\em Classification of Fano threefolds with $B_2\geq 2$,} Manuscripta Math. 36 (2) (1981/82), no.2, pp. 147-162. 
  
  \bibitem{YPG} M. Reid, {\em Young Person's Guide to Canonical Singularities,} Proc. Sympos. Pure Math. 46, Providence, R.I.: Amer. Math. Soc., pp. 345-414.
 
 \bibitem{Sho} V.V. Shokurov, {\em Smoothness of a general anticanonical divisor on a Fano variety,} Izv. Akad. Nauk SSSR Ser. Mat 43 (1979), no. 2, pp. 430-441.
  
  \end{thebibliography}
\end{document}